\documentclass{aptpub}
\authornames{Xiaoyu Lei} 

\shorttitle{An Efficient Method To Generate A Discrete Uniform Distribution Using A Biased Random Source} 

\usepackage[margin=1in]{geometry}
\usepackage{amsmath}
\usepackage{graphicx}
\usepackage[noend]{algpseudocode}
\usepackage{algorithmicx,algorithm}
\usepackage{appendix}
\usepackage{authblk}
\usepackage{verbatim}
\usepackage{amssymb}

\newcommand{\probability}{\mathbb{P}}
\newcommand{\choosenm}[2]{\binom{#1}{#2}}
\newcommand{\algone}{\mathcal{A}_1}

\newcommand{\flips}[1]{\boldsymbol{X}^{#1}}
\newcommand{\algtwo}[1]{\mathcal{A}_2(#1)}
\newcommand{\id}{\operatorname{id}}
\newcommand{\algthree}[1]{\mathcal{A}_3(#1)}
\newcommand{\Shead}[1]{S_{\text{head}}(#1)}
\newcommand{\Nhead}[1]{N_{\text{head}}(#1)}

\begin{document}

\title{An Efficient Method To Generate A Discrete Uniform Distribution Using A Biased Random Source}

\authorone[The University of Chicago]{Xiaoyu Lei} 
\addressone{5747 South Ellis Avenue, Chicago, Illinois, USA} 
\emailone{leixy@uchicago.edu} 

\begin{abstract}
This article presents an efficient algorithm to generate a discrete uniform distribution on a set of $p$ elements using a biased random source for $p$ prime. The algorithm generalizes Von Neumann's method and improves computational efficiency of Dijkstra's method. In addition, the algorithm is extended to generate discrete uniform distribution on any finite set based on the prime factorization of integers. The time complexity of the proposed algorithm is overall sublinear $\operatorname{O}(n/\log n)$.
\end{abstract}



\section{Background} 


Sampling a target distribution from a random physical source has many applications. However, the random physical sources are often biased with unknown distribution, while we need a specific target distribution in applications. Therefore, an efficient algorithm generating target distribution from a random source is of great value. \cite{von195113} firstly proposed a simple method to generate a fair binary distribution from an unfair binary source with an unknown bias. His method has served as a precursor of a series of algorithms generating a target distribution from an unknown random source.

\cite{hoeffding1994unbiased} and \cite{stout1984tree} improved Von Neumann's method to generate a fair binary distribution from a biased random source. From the view of probability theory, \cite{elias1972efficient} formally defined the kind of random procedure that can generate a target distribution. Elias also designed an infinite sequence of sampling schemes, with computational efficiency decreasing to the theoretical lower bound. Elias did not provide an executable algorithm for his method. Elias' method needs to generate Elias' function first. While such a preprocessing step needs an exponential space cost and at least a polynomial time cost \cite{pae2005random}. Thus Elias' method is computationally costly and inefficient.

\cite{dijkstra1990making} provided another method of generating uniform distribution on a set of $p$ elements for $p$ prime, while Dijkstra's method is computationally inefficient. Indeed, when realizing his method, we need a preprocessing step to generate and store a function which maps outcomes from the random source to some target values. However, such a preprocessing step needs an exponential time and space cost.

In this article, we propose a new algorithm based on the idea of Dijkstra's method. The proposed algorithm does not need a preprocessing step, thus computationally efficient.

This article is organized as follows: In Section \ref{sec-2}, we briefly recast Von Neumann's method as a starting point as well as a special case of our algorithm. In Section \ref{sec-3}, we heuristically construct and explain our algorithm. In Sections \ref{sec-4} and \ref{sec-5}, we formally propose our algorithms and theoretically verify them. In Section \ref{sec-5}, we prove that our algorithm has overall sublinear time complexity. Another novel proof of Theorem \ref{thm-3-1} is given in Appendix \ref{App-A}.

\section{Introduction to Von Neumann's Method}\label{sec-2}
Let $X\in \{ H, T \}$ denote the outcome of a biased coin flip with probability $a=\probability(X=H)\in (0,1)$ getting a head and probability $b=\probability(X=T)=1-a$ getting a tail. Let $\{X_i:i\ge 0\}$ be i.i.d. copies of $X.$ Von Neumann proposed an algorithm $\algone$ generating a fair binary random variable with distribution $\probability(\algone=0)=\probability(\algone=1)=1/2$ in the following way \cite{von195113}:
\begin{algorithm}[ht]
\caption{$\algone$: Von Neumann's Algorithm Generating Fair Binary Random Variable}
\hspace*{0.02in} {\bf Input:}
A sequence of flips from a biased coin $X$\\
\hspace*{0.02in} {\bf Output:}
Integer 0 or 1
\begin{algorithmic}[1]
\State Flip the coin twice
\State If the result is either HH or TT, then discard the two coin flips and return to step 1
\State If the result is HT, \Return $\algone=0$. If the result is TH, \Return $\algone=1$
\end{algorithmic}
\end{algorithm}

Let $\{ \boldsymbol{Y}_i = (X_{2i}, X_{2i+1}): i \ge 0 \}$ be i.i.d. outcomes of pairs of flips and $\tau$ be the first time such that $\boldsymbol{Y}_i \in \{ HT, TH \}$, then we will have
\begin{equation*}
    \probability(\algone=0)  = \probability(\boldsymbol{Y}_\tau = HT) =\frac{\probability(\boldsymbol{Y}_0 = HT)}{\probability(\boldsymbol{Y}_0\in \{ HT, TH \})}= \frac{\probability((X_0, X_1)=HT)}{\probability((X_0, X_1)\in \{ HT, TH \})}=\frac{1}{2}.
\end{equation*}

The derivation above shows $\algone$ generates a fair binary distribution. Below, we propose an efficient algorithm to generate a uniform distribution on $p$ elements for a prime $p$. At each cycle, we flip a coin $p$ times, the algorithm returns a number in $\{ 0, \cdots, p-1 \}$ except when the $p$ flips are all heads or all tails, analogous to Von Neumann's method.

\section{Heuristic Explanation for The Main Idea}\label{sec-3}
Let random vector
\begin{equation*}
    \flips{n}= ( X_0, \cdots, X_{n-1} )\in \{H, T \}^n
\end{equation*}
be the outcome of $n$ flips. Let $N_{\text{head}}(\flips{n})$ denote the  head count in $\flips{n}$, and $S_{\text{head}}(\flips{n})$ denote the rank sum of heads in $\flips{n}$, with ranks ranging from $0$ to $n-1$,
\begin{equation}\label{def_N_S}
    N_{\text{head}}(\flips{n})=\sum_{i=0}^{n-1}1_{\{X_i = H \}}\quad\text{and}\quad S_{\text{head}}(\flips{n})=\sum_{i=0}^{n-1} i\cdot 1_{\{X_i = H \}}.
\end{equation}
For example, when $\flips{5}=(H,H,T,H,T)$, we have $N_{\text{head}}(\flips{5})=3$ and $S_{\text{head}}(\flips{5})=4.$

For a specific sequence of $n$ flips $\boldsymbol{x}^n=(x_0,\cdots, x_{n-1})\in \{H, T \}^n$ as an observation of $\flips{n}$, if $N_{\text{head}}(\boldsymbol{x}^n)=\sum_{i=0}^{n-1}1_{\{x_i=H\}}=k$, then the probability of getting $\boldsymbol{x}^n$ in $n$ flips is
\begin{equation*}
    \probability(\flips{n} = \boldsymbol{x}^n) = \prod_{i=0}^{n-1}\probability(X_i = x_i)=a^k b^{n-k},
\end{equation*}
which only depends on the head count $k$. As a result, for $0\le k\le n$, there are exactly $\choosenm{n}{k}$ outcomes of $n$ flips containing $k$ heads, each with the same probability $a^k b^{n-k}$. Let
\begin{equation}\label{S_k}
    S_k = \left\{ A\subset \{0,1,\cdots, n-1\}:|A|=k \right\},
\end{equation}
where $|A|$ means the cardinality of set $A$. Thus $S_k$ is the set of all subsets of $\{ 0, \cdots, n-1 \}$ containing $k$ elements. Note that $|S_k|=\choosenm{n}{k}$ and each element in $S_k$ corresponds to one and only one outcome of $n$ flips with $k$ heads in the following way
\begin{equation}\label{1-1}
   \{i_1,\cdots, i_k\}\in S_k \quad\longleftrightarrow \quad \mathop{\cdots H\cdots H\cdots H\cdots}\limits_{i_1\hspace{7mm} i_2 \hspace{1mm}\cdots\hspace{2mm}  i_k},
\end{equation}
where each $i_t$ corresponds to the rank of an appearance of head in the $i_t$-th flip of $n$ flips, $i_1 < i_2 <\cdots <i_k$. As a result, we have the one-to-one correspondence below
\begin{equation}\label{1-1relation}
    S_k \quad\longleftrightarrow\quad \{ \boldsymbol{x}^n \in \{H,T\}^n :N_{\text{head}}(\boldsymbol{x}^n)=k \},
\end{equation}
and we also have
\begin{equation*}
    \probability(\flips{n} = \boldsymbol{x}^n)=a^{k}b^{n-k}, \quad \forall \, \boldsymbol{x}^n \in S_k.
\end{equation*}
Note for the correspondences \eqref{1-1} and \eqref{1-1relation}, we do not distinguish the left side and right side in the derivation below. And the equivalences will be frequently used in the following proof.

Inspired by Von Neumann's algorithm, we consider an algorithm generating a distribution on the set $\{ 0, \cdots, n-1 \}.$ At each cycle, we flip the coin $n$ times, then the algorithm returns a number in $\{ 0, \cdots, n-1 \}$ except when the outcome is all heads or all tails. Define sets $\{ A_m : 0\le m\le n-1\}$ to be a disjoint partition of $\bigsqcup_{1\le k\le n-1}S_k$,
\begin{equation*}
    \bigsqcup_{k=1}^{n-1}S_k = \bigsqcup_{m=0}^{n-1}A_m,
\end{equation*}
where $\bigsqcup$ means disjoint union.
The algorithm is formally stated below.
\begin{algorithm}[ht]
\caption{$\mathcal{A}$: Generating A Discrete Distribution on Set $\{0, \cdots, n-1 \}$}
\hspace*{0.02in} {\bf Input:}
A number $n$, a sequence of flips from a biased coin $X$\\
\hspace*{0.02in} {\bf Output:}
Integer in $\{0, \cdots, n-1 \}$
\begin{algorithmic}[1]
\State Flip the coin $n$ times, denote the outcome by $\flips{n}\in \{H,T\}^{n}$
\State If the result is either all heads or all tails, then discard the outcome and return to step 1
\State Else \Return $m$ when $\flips{n} \in A_m$
\end{algorithmic}
\end{algorithm}

Let $\{\boldsymbol{Y}_i = (X_{in}, \cdots, X_{in+n-1}): i\ge 0\}$ be i.i.d. outcomes of $n$ flips and $\tau$ be the first time $\boldsymbol{Y}_i$ is neither all heads nor all tails. Then for $0 \le m \le n-1$, we have
\begin{align}\label{3-1}
    \probability(\mathcal{A}=m)
    & = \probability(\boldsymbol{Y}_\tau \in A_m) \notag \\
    & = \frac{\probability(\flips{n} \in A_m )}{\probability (\flips{n} \in S_k \text{ for some } 1\le k\le n-1)} \notag \\
    & = \frac{\sum_{k=1}^{n-1} \probability (\flips{n} \in A_m\cap S_k) }{\sum_{k=1}^{n-1} \probability (\flips{n} \in S_k) } \notag \\
    & = \frac{\sum_{k=1}^{n-1} |A_m \cap S_k| a^k b^{n-k} }{\sum_{k=1}^{n-1} |S_k| a^k b^{n-k} }.
\end{align}

Let us consider a special case of the algorithm above, where $n$ is a prime $p.$ The reason for focusing on prime $p$ comes from the following fact in number theory,
\begin{equation*}
    p \Big| \choosenm{p}{k} = |S_k|, \quad \forall \, 1\le k\le p-1
\end{equation*}
where the symbol $|$ means ``divides". Then for each $k$, we can partition $S_k$ into disjoint $p$ parts of equal size. For $1\le k \le p-1,$ assume that the choice of sets $\{ A_m: 0 \le m \le p-1 \}$ satisfies
\begin{equation}\label{3-2}
    |A_0 \cap S_k| = \cdots = |A_{p-1} \cap S_k| = \frac{1}{p}|S_k|,
\end{equation}
where the disjoint $\{ A_m \cap S_k: 0\le m\le p-1 \}$ partition $S_k$ into $p$ subsets of equal size. Based on \eqref{3-1} and \eqref{3-2}, for $0 \le m\le p-1,$ we have
\begin{equation*}
    \probability(\mathcal{A}=m)= \frac{\sum_{k=1}^{p-1} |A_m \cap S_k| a^k b^{n-k} }{\sum_{k=1}^{p-1} |S_k| a^k b^{n-k} } = \frac{\sum_{k=1}^{p-1} \frac{1}{p}|S_k| a^k b^{n-k} }{\sum_{k=1}^{p-1} |S_k| a^k b^{n-k} } = \frac{1}{p},
\end{equation*}
which means the algorithm $\mathcal{A}$ returns a uniform distribution on $\{ 0, \cdots, p-1 \}.$

What remains is to find $\{ A_m: 0\le m\le p-1 \}$ satisfying \eqref{3-2}. We can always first partition $S_k$ into $p$ subsets of equal size, and then define $\{ A_m \cap S_k: 0\le m\le p-1 \}$ to be these subsets, like the proposed method in \cite{dijkstra1990making}. However, there exist two disadvantages of this method. First, everyone can have his way of partitioning $S_k$ into subsets of equal size, and there is no widely accepted standard. Second, partitioning $\{S_k:1\le k\le p-1\}$ and designing $\{ A_m:0\le m\le p-1 \}$ need excessive time and storage cost, because there are $2^p$ different outcomes of $p$ flips we need to handle, which grows exponentially as $p$ increases. A preprocessing step of exponential time is unacceptable for an efficient algorithm.

With the help of the modulo $p$ function, there exists an ingenious way of designing $\{ A_m:0\le m\le p-1 \}$ to satisfy \eqref{3-2}. Based on the correspondence \eqref{1-1}, for $0\le m\le p-1,$ indeed, we can choose
\begin{equation}\label{A_m}
    A_m = \left\{  \flips{p}: \Shead{\flips{p}} = m \mod{p}  \right\},
\end{equation}
as we will show in the next section.

\section{Generating Uniform Distribution on $p$ (Prime) Elements}\label{sec-4}
We give an algorithm generating discrete uniform distribution on the set $\{0, \cdots, p-1  \},$ where $p$ is a prime.
\begin{algorithm}[ht]\label{alg-2}
\caption{$\algtwo{p}$: Generating Discrete Uniform Distribution on Set $\{0, \cdots, p-1 \}$}
\hspace*{0.02in} {\bf Input:}
A prime number $p$, a sequence of flips from a biased coin $X$\\
\hspace*{0.02in} {\bf Output:}
Integer in $\{0, \cdots, p-1 \}$
\begin{algorithmic}[1]
\State Flip the coin $p$ times, denote the outcome by $\flips{p}\in \{H,T\}^{p}$
\State If the result is either all heads or all tails, then discard the outcome and return to step 1
\State Else \Return $\Shead{\flips{p}}\mod{p}$
\end{algorithmic}
\end{algorithm}

We need the following lemma before proving the main theory.
\begin{lemma}\label{lemma}
Let $p$ be a prime number, let $\{ S_k: 1\le k\le p-1 \}$ consist of all subsets of $\{ 0,\cdots, p-1 \}$ having $k$ elements. For fixed $k$, let $\{ S_k^m : 0\le m\le p-1 \}$ be defined by
\begin{equation}\label{S^m_k}
    S^m_k = \left\{\{i_1,\cdots, i_k\}\in S_k:\sum_{j=1}^k i_j = m \mod{p} \right\}.
\end{equation}
Note that $S^m_k = A_m \cap S_k$, where $A_m$ is defined in \eqref{A_m}.

Then we have
\begin{equation*}
    \left| S^m_k \right| = \frac{1}{p}\choosenm{p}{k}, \quad \forall \, 1\le k\le p-1, \,\forall \, 0\le m\le p-1.
\end{equation*}
\end{lemma}
\begin{proof}
For fixed $1\le k\le p-1$, consider a permutation on $S_k$ defined in the following way,
\begin{equation*}
f(\{i_1,\cdots, i_k\})=\{(i_1+1)\mod{p}, \cdots, (i_k+1)\mod{p}\}.
\end{equation*}
Denote $f^0$ to be the identity function $\id$. Let $\langle f\rangle$ be the subgroup generated by $f$. We need to show
\begin{equation*}
    \langle f\rangle = \{f^0=\id, f^1,\cdots f^{p-1}\}.
\end{equation*}
Since we know $f^{p}=\id$, we need to show $f^s\neq \id$ for $1\le s\le p-1.$

If $f^s= \id$ for some $1\le s\le p-1$, then we have
\begin{equation*}
    f^s(\{i_1,\cdots, i_k\})=\{(i_1+s)\mod{p}, \cdots, (i_k+s)\mod{p}\}=\{i_1,\cdots, i_k\},
\end{equation*}
from which we have
\begin{equation*}
    \sum_{j=1}^k(i_j+s) = \sum_{j=1}^k i_j \mod{p}.
\end{equation*}
The equality above shows $ p|ks$, which implies $p|k \text{ or } p|s$, leading to a contradiction since $1\le k,s\le p-1.$

Let group $\langle f\rangle$ act on $S_k.$ For $\{ i_1,\cdots,i_k \}\in S_k$, let $O_{\{ i_1,\cdots,i_k \}}$ denote the orbit of $\{ i_1,\cdots,i_k \}$ under group action
\begin{equation*}
    O_{\{ i_1,\cdots,i_k \}} = \{ \{ i_1^s,\cdots,i_k^s \}:=f^s(\{ i_1,\cdots,i_k \}) ,\text{ for } 0\le s\le p-1 \}.
\end{equation*}
The theory of group action tells us that $S_k$ can be divided to disjoint orbits with equal size $p$. In addition, for any $\{ i_1,\cdots,i_k \}\in S_k,$ when $s$ varies from $0$ to $p-1$,
\begin{equation*}
    \sum_{j=1}^k i^s_j \mod{p}
\end{equation*}
takes all values in $\{0,\cdots, p-1\}.$

If the claim above were not true, then there would exist $0\le s_1<s_2 \le p-1$ such that
\begin{equation*}
    \sum_{j=1}^k i^{s_1}_j = \sum_{j=1}^k i^{s_2}_j \mod{p}\quad  \Rightarrow\quad  \sum_{j=1}^k(i_j+s_1) = \sum_{j=1}^k (i_j+s_2) \mod{p}.
\end{equation*}
The equality above shows $p|k(s_2-s_1)$, which implies $p|k \text{ or } p|(s_2-s_1)$, leading to a contradiction since $1\le k,s_2-s_1\le p-1.$

The proof above shows that $S_k$ is a union of disjoint orbits of equal size $p.$ And in each orbit, for $0\le m\le p-1$, there exists one and only one element belonging to $S^m_k$, which means $\{ S_k^m : 0\le m\le p-1 \}$ partition $S_k$ into $p$ subsets with equal size and the proof is complete.

\end{proof}

The following is a special case to show the idea of the proof, with $p=7$ and $k=3$, the proof will process as the table shows.
\begin{figure}[H]
    \centering
    \includegraphics[width=12cm]{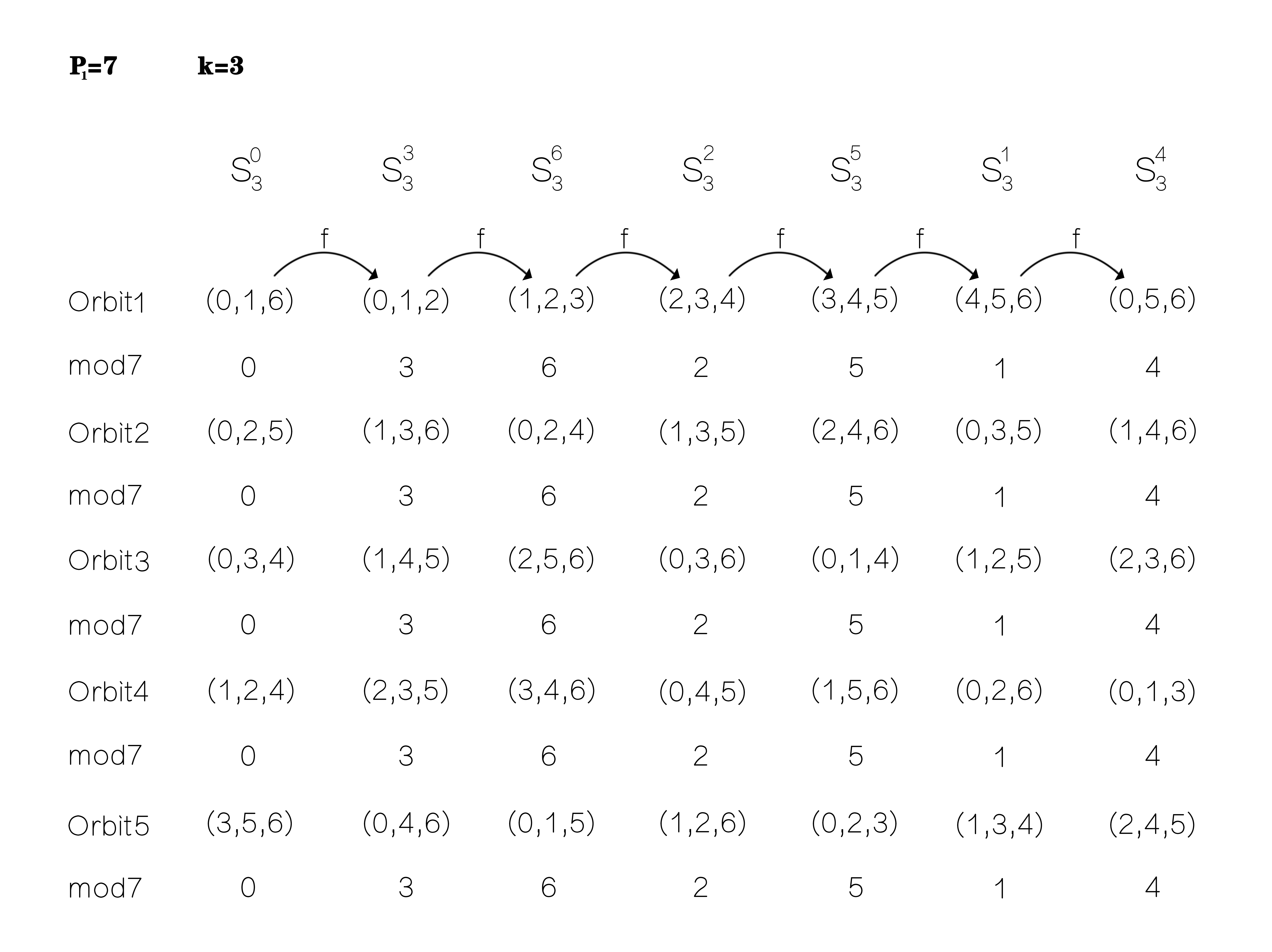}
    \caption{An example of the method in the proof}
\end{figure}

Next, we prove the main theorem on algorithm $\algtwo{p}.$
\begin{thm}\label{thm-3-1}
Let $X$ denote a biased coin with probability $a \in (0, 1)$ of getting a head and probability $b = 1 - a$
of getting a tail. For a prime $p$, $\algtwo{p}$ has the following properties:

(i) $\algtwo{p}$ terminates in finite number of flips with probability 1. The algorithm returns a uniform distribution on $\{0,\cdots,p-1\},$
\begin{equation*}
    \probability(\algtwo{p}=m)=\frac{1}{p}, \quad \forall\, 0\le m\le p-1.
\end{equation*}

(ii) The expected number of flips terminating $\algtwo{p}$ is
\begin{equation*}
    \frac{p}{1-a^{p}-b^{p}},
\end{equation*}
which means when $p$ is large, the time complexity approximates to the linear $\operatorname{O}(p).$

(iii) By letting $p=2$, $\algtwo{2}$ is exactly the Von Neumann's algorithm $\mathcal{A}_1.$
\end{thm}

\begin{proof}
Let $\flips{p}=(X_0, \cdots, X_{p-1})$ be the outcome of $p$ flips of a biased coin, a random variable taking values in $\{H,T\}^{p}.$ Based on the correspondences \eqref{1-1} and \eqref{1-1relation}, and the definition of $S^m_k$ in \eqref{S^m_k}, each $\boldsymbol{x}^p \in \{H,T\}^{p}$ corresponds to one and only one element in $S^m_k$ by $\Nhead{\boldsymbol{x}^p}=k$ and $\Shead{\boldsymbol{x}^p}=m\mod{p}$ for some $k$ and $m$, where $\Nhead{\boldsymbol{x}^p}$ and $\Shead{\boldsymbol{x}^p}$ in \eqref{def_N_S} are the count and rank sum of heads respectively. Recall the definition of $S_k$ in \eqref{S_k}, then by Lemma \ref{lemma}, $\{ S_k^m : 0\le m\le p-1 \}$ partition $S_k$ into $p$ subsets with equal size.

Let $\{\boldsymbol{Y}_i = (X_{ip}, \cdots, X_{ip+p-1}): i\ge 0\}$ be i.i.d. outcomes of $p$ flips and $\tau$ be the first time $\boldsymbol{Y}_i$ is neither all heads nor all tails. Then for $0 \le m \le p-1$, we have
\begin{align*}
    \probability(\algtwo{p}=m)
    & = \probability(\Shead{\boldsymbol{Y}_\tau} = m \mod{p})  \\
    & = \probability(\Shead{\flips{p}} = m \mod{p} | \flips{p}\text{ is neither all heads nor all tails}) \\
    & = \frac{\probability(\Shead{\flips{p}}=m \mod{p}, N_{\text{head}}(\flips{p}) =k \text{ for some } 1\le k\le p-1)}{\probability(N_{\text{head}}(\flips{p}) =k \text{ for some } 1\le k\le p-1)} \\
    & = \frac{\sum_{k=1}^{p-1}\probability(\Shead{\flips{p}}=m \mod{p}, N_{\text{head}}(\flips{p}) =k)}{\sum_{k=1}^{p-1}\probability( N_{\text{head}}(\flips{p}) =k )} \\
    & = \frac{\sum_{k=1}^{p-1}|S_k^m|a^k b^{p-k}}{\sum_{k=1}^{p-1}|S_k|a^k b^{p-k}} \\
    & = \frac{1}{p},
\end{align*}
where the last identity is implied by the fact that $|S^m_k| = \frac{1}{p}\choosenm{p}{k}=\frac{1}{p}|S_k|$.

Let $E$ denote the expected number of flips terminating $\algtwo{p}$. Hence $E$ satisfies the following equation
\begin{equation*}
    E = p\probability(N_{\text{head}}(\flips{p}) =k \text{ for some } 1\le k\le p-1) + (p+E)\probability(\flips{p}\text{ is all heads or all tails}),
\end{equation*}
from which we have
\begin{equation*}
    E = \frac{p}{1-\probability(\flips{p}\text{ is all heads or all tails})}=\frac{p}{1-a^{p}-b^{p}}.
\end{equation*}
\end{proof}

We also came up with a creative and short proof for Theorem \ref{thm-3-1} (i) using random variables in residue class $\mathbb{Z}_{p}$ See Appendix \ref{App-A} for the new proof.

\section{Generating Uniform Distribution on $n$ Elements}\label{sec-5}
Denote $n$ to be any positive integer with prime factorization $n=\prod_{i=1}^s p_i^{t_i}.$ Let $\mathcal{M}$ be the set of all prime factors of $n$ considering multiplicity, which means $p_i$ appears $t_i$ times in $\mathcal{M}.$ The following algorithm $\algthree{n}$ generates discrete uniform distribution on the set $\{0,\cdots, n-1\}$ in an iterative way.

\begin{algorithm}[ht]\label{alg-3}
\caption{$\algthree{n}$: Generating Discrete Uniform Distribution on Set $\{0, \cdots, n - 1 \}$}
\hspace*{0.02in} {\bf Input:}
A sequence of flips, an integer $n$, a set $\mathcal{M}$ containing all prime factors of $n$, where each prime repeats as many times as its multiplicity in the decomposition of $n$\\
\hspace*{0.02in} {\bf Output:}
Integer in $\{0, \cdots, n-1 \}$
\begin{algorithmic}[1]
\State Set $r=0$
\While{$\mathcal{M}\neq \emptyset$}
\State Take a prime $p^{\prime}$ out of $\mathcal{M}$
\State $n = n/p^{\prime}$
\State Run $\algtwo{p^\prime}$, and let $t$ denote the return value
\State $r = r + t \cdot n$
\EndWhile
\State \Return r
\end{algorithmic}
\end{algorithm}

The following theorem shows the validity of algorithm $\algthree{n}$.

\begin{thm}\label{thm-5-1}
For any integer $n$, $\algthree{n}$ has the following properties:

(i) $\algthree{n}$ terminates in finite number of flips with probability 1. It returns a uniform distribution on $\{0, \cdots, n-1\}$
\begin{equation*}
    \probability(\algthree{n}=m)=\frac{1}{n}, \quad \forall \, 0\le m\le n-1.
\end{equation*}

(ii) When $n$ has prime factorization $\prod_{i=1}^s p_i^{t_i}$, the expected number of flips terminating $\algthree{n}$ is
\begin{equation*}
    \sum_{i=1}^s \frac{t_i p_i}{1-a^{p_i} - b^{p_i}}.
\end{equation*}
Therefore, the time complexity is approximately $\sum_{i=1}^s t_i p_i$ for large $n.$

(iii) The overall order of time complexity is $\operatorname{O}(n/\log(n)).$
\end{thm}

\begin{proof}
To show the claim (i), note that each outcome of $\algthree{n}$ corresponds to one and only one sequence of outcomes of $\algtwo{p_i}$. For this fact, first we consider a simplified case where $n=p_1p_2$ is a product of two prime numbers $p_1$ and $p_2$, and $p_1$ may equal $p_2$.

Given $n=p_1 p_2$, then $\mathcal{M} = \{p_1, p_2 \}$. Suppose we first get $p_1$ from $\mathcal{M}$ and then $p_2$. Then the outcomes $\algtwo{p_1}=m_1$ and $\algtwo{p_2}=m_2$ correspond to the outcome $\algthree{n}=m_1 p_2 + m_2$. Since $0\le m_1\le p_1-1$ and $0\le m_2\le p_2-1$, we have the range for $\algthree{n}$:
$$
0\le \algthree{n}\le (p_1-1)p_2 + p_2 -1 = n-1,
$$
which shows the fact $\algthree{n}\in \{0, \cdots, n-1\}.$ Note that for $0\le m\le n-1$, there exists one and only one pair of $(m_1, m_2)$ as
$$
\left(\left\lfloor \frac{m}{p_2} \right\rfloor, m-\left\lfloor \frac{m}{p_2} \right\rfloor p_2\right)
$$
satisfying the equation $m = m_1p_2+m_2\, (0\le m_1\le p_1-1, \, 0\le m_2\le p_2-1).$ So the outcome $\algthree{n} = m$ corresponds to the outcomes $\algtwo{p_1}=m_1$ and $\algtwo{p_2}=m_2$.

For the general case $n=\prod_{i=1}^s p_i^{t_i}$, based on the same method above, we conclude that for each $m$, there exists a unique set $\{ m_{p^\prime}: p^\prime\in \mathcal{M}\}$ such that the outcome $\algthree{n} = m$ corresponds to the outcomes $\algtwo{p^\prime}=m_{p^\prime}\,(p^\prime \in \mathcal{M})$. Therefore, the probability of $\algthree{n}=m$ is
\begin{equation*}
    \probability(\algthree{n}=m)=\prod_{p^\prime\in\mathcal{M}}\probability(\algtwo{p^\prime}=m_{p^\prime}) 
    = \prod_{i=1}^s \left(\frac{1}{p_i}\right)^{t_i}=\frac{1}{n}, \quad \forall \, 0\le m\le n-1.
\end{equation*}

To prove the claim (ii), note for $n=\prod_{i=1}^s p_i^{t_i},$ the set $\mathcal{M}$ contains each prime factor $p_i$ with $t_i$ times. By the iterative construction of $\algthree{n}$, we need to run $\algthree{p_i}$ once every time we pick $p_i$ from $\mathcal{M}$. Based on (ii) of Theorem \ref{thm-3-1}, the expected number of flips for $\algtwo{p_i}$ is $\frac{p_i}{1 - a^{p_i} - b^{p_i}}$, from which we conclude the expected number of flips terminating $\algthree{n}$ is
$$
    \sum_{i=1}^s \frac{t_i p_i}{1-a^{p_i} - b^{p_i}}.
$$

To analyze the time complexity of the algorithm $\algthree{n}$, define the function $c(n)=\sum_{i=1}^s t_i p_i$ to be the sum of prime factors of $n$ multiplied by their multiplicity, which is a good approximation to the time complexity of $\algthree{n}$ according to Theorem \ref{thm-5-1} (ii). We see that for prime numbers, the complexity is linear. For composite numbers, the complexity is sublinear. For $n=p_1^{t_1}$, since $c(n)=t_1 p_1$, the time complexity is almost $\log(n).$ We have the following theorem from number theory,
\begin{equation*}
    \lim_{N \to \infty}\left. \left|\left\{ 2\le n\le N: c(n)<\frac{n}{\log^{1-\epsilon}(n)}\right\}\right| \middle/ N = 1 \right. ,\quad \forall \, 0 < \epsilon < 1,
\end{equation*}
according to Corollary 2.11 of \cite{jakimczuk2012sum}. So we have an overall sublinear $\operatorname{O}(n/\log(n))$ complexity for the algorithm $\algthree{n}.$

\end{proof}

\begin{remnn}
In \cite{elias1972efficient}, another method generating  discrete uniform distribution on the set $\{0,\cdots ,n - 1\}$ was proposed. Elias' method needs Elias' function mapping outcomes of the random source to target values. However, unlike Theorem \ref{thm-5-1} (iii), the efficiency of Elias' method is defined by complicated mathematical formulas without analytic and concise form, which is hard to analyze theoretically. Besides, Elias' method suffers the same problem as Dijkstra's method mentioned in Section \ref{sec-3}. The computation of Elias' function, an essential preprocessing step of Elias' method, is computationally inefficient, and the storage of Elias' function is also an excessive space cost.
\end{remnn}









\appendix

\section{A New Proof for Theorem \ref{thm-3-1} (i)}\label{App-A}
Consider random variables taking values in $\mathbb{Z}_p=\{ \bar{0}, \cdots, \overline{p-1} \}$, where $\overline{i}$ represents the residual class of $i$ modulo $p$. Regard $\bar{0}$ as a tail and $\bar{1}$ as a head. Let $X$ denote the outcome of a flip satisfying $\probability(X=\overline{0})=a$ and $\probability(X=\overline{1})=b$. Let $X_0, \cdots, X_{p-1}$ be independent copies of $X$. Define $\flips{p}=(X_0,\cdots, X_{p-1})$ to be the outcome of $p$ flips. We then have the following two equivalences,
\begin{equation*}
    \flips{p}\text{ is all heads or all tails}\Longleftrightarrow
    X_i=\bar{0}\,( \forall 0\le i\le p-1)\text{ or }
    X_i=\bar{1}\,( \forall 0\le i\le p-1)
    \Longleftrightarrow \sum_{i=0}^{p-1}X_i=\bar{0},
\end{equation*}
and
\begin{equation*}
    S_{\text{head}}(\flips{p}) \mod p =m   \Longleftrightarrow \sum_{i=0}^{p-1}\overline{i}\cdot X_i = \overline{m}.
\end{equation*}

Also note for any permutation $\sigma$, we have
\begin{equation*}
    (X_0,\cdots, X_{p-1}) \overset{d}{=} (X_{\sigma(0)},\cdots, X_{\sigma(p-1)}),
\end{equation*}
since all $X_i$'s are i.i.d.. In the following, we let $\sigma$ denote the special permutation
\begin{equation*}
    \sigma = 
\left(
\begin{array}{ccccc}
    0 & 1 & \cdots & p-2 & p-1 \\
    1 & 2 & \cdots & p-1  & 0
\end{array}
\right).
\end{equation*}

For fixed $t\neq \bar{0}\in \mathbb{Z}_{p}$, we have
\begin{align*}
    \probability\left(\sum_{i=0}^{p-1}\overline{i}\cdot X_i=\bar{0} ,\quad \sum_{i=0}^{p-1}X_i=t\right)
    & = \probability\left(\sum_{i=0}^{p-1}\overline{i}\cdot X_i + \sum_{i=0}^{p-1}X_i=t ,\quad \sum_{i=0}^{p-1}X_i=t\right) \\
    & = \probability\left(\sum_{i=0}^{p-1}\overline{i+1}\cdot X_i=t ,\quad \sum_{i=0}^{p-1}X_i=t\right) \\
    & = \probability\left(\sum_{i=0}^{p-1}\overline{i+1}\cdot X_{\sigma(i)}=t ,\quad \sum_{i=0}^{p-1}X_{\sigma(i)}=t\right) \\
    & = \probability\left(\sum_{i=0}^{p-1}\overline{i}\cdot X_i=t ,\quad \sum_{i=0}^{p-1}X_i=t\right).
\end{align*}

Note any $t\neq \bar{0}$ can generate $\mathbb{Z}_{p}$. By iterating the derivation above, we have
$$
\probability\left(\sum_{i=0}^{p-1}\overline{i}\cdot X_i=k ,\quad \sum_{i=0}^{p-1}X_i= t \right) = 
\probability\left(\sum_{i=0}^{p-1}\overline{i}\cdot X_i=s ,\quad \sum_{i=0}^{p-1}X_i= t \right),\quad \forall\, k,s \in \mathbb{Z}_{p}.
$$

Summing over $t \neq \bar{0}$ on both sides of the above equation, we have for $k,s\in \mathbb{Z}_{p}$
\begin{align*}
\probability\left(\sum_{i=0}^{p-1}\overline{i}\cdot X_i=k ,\quad \sum_{i=0}^{p-1}X_i\neq \bar{0}\right)
& = \sum_{t\neq \bar{0}}\probability\left(\sum_{i=0}^{p-1}\overline{i}\cdot X_i=k ,\quad \sum_{i=0}^{p-1}X_i=t\right) \\
& =  \sum_{t\neq \bar{0}}\probability\left(\sum_{i=0}^{p-1}\overline{i}\cdot X_i=s ,\quad \sum_{i=0}^{p-1}X_i=t\right) \\
& = \probability\left(\sum_{i=0}^{p-1}\overline{i}\cdot X_i=s ,\quad \sum_{i=0}^{p-1}X_i\neq \bar{0}\right),
\end{align*}
which implies for $k,s\in \mathbb{Z}_{p},$
\begin{equation*}
    \probability\left(\sum_{i=0}^{p-1}\overline{i}\cdot X_i=k \left| \quad \sum_{i=0}^{p-1}X_i\neq \bar{0}\right.\right)
    = 
    \probability\left(\sum_{i=0}^{p-1}\overline{i}\cdot X_i=s \left| \quad \sum_{i=0}^{p-1}X_i\neq \bar{0}\right.\right).
\end{equation*}

The equality above is equal to the statement
\begin{align*}
    & \quad\,\, \probability(S_{\text{head}}(\flips{p})=k \mod p|\flips{p}\text{ is neither all heads nor all tails}) \\
    & =\probability(S_{\text{head}}(\flips{p})=s\mod p|\flips{p}\text{ is neither all heads nor all tails}), \quad \forall\, 0\le k,s\le p-1,
\end{align*}
as desired.



\acks 
\noindent The author appreciates Prof. Mei Wang at UChicago for helpful discussions and advice. The author thanks Ph.D. candidate Haoyu Wei at UCSD for useful suggestions and kind support. The author also appreciates the editor of \textit{Journal of Applied Probability} and the two anonymous referees for their valuable comments and remarks.

\fund 
\noindent There are no funding bodies to thank relating to this creation of this article.

\competing 
\noindent There were no competing interests to declare which arose during the preparation or publication process of this article.


\end{document}